\documentclass[a4paper,11pt,reqno]{amsart}
\usepackage{amsmath,a4,amsfonts,amssymb,amsthm,enumerate,fleqn,breqn,dsfont}
\usepackage{enumitem}
\setdescription{leftmargin=0cm,labelindent=0cm}
\usepackage[a4paper,left=1in, right=1in, top=1.1in, bottom=1.1in]{geometry}

\usepackage{etoolbox}
\patchcmd{\section}{\scshape}{\bfseries}{}{}
\makeatletter
\renewcommand{\@secnumfont}{\bfseries}
\makeatother

\def\RR{\mathbb{R}}
\def\NN{\mathbb{N}}
\def\HH{\mathfrak{Q}}

\def\T{\mathfrak{T}}
\def\cc#1{\{#1\}}
\def\pp#1{\|#1\|}
\def\ra{\rangle}
\def\la{\langle}
\def\nh{\mathcal{H}}

\def\H{V_R(\HH)}
\def\K{W_R(\HH)}

\def\cc#1{\{#1\}}
\def\pp#1{\|#1\|}
\def\ra{\rangle}
\def\la{\langle}

\theoremstyle{plain}
\newtheorem{theorem}{\bf Theorem}[section]

\theoremstyle{remark}
\newtheorem{definition}[theorem]{\bf Definition}

\newtheorem{corollary}[theorem]{\bf Corollary}

\title[R\lowercase{iesz bases in quaternionic} H\lowercase{ilbert Spaces}]{Riesz bases  in quaternionic Hilbert Spaces}
\author[ S\lowercase{harma}, V\lowercase{irender} \lowercase{and} K\lowercase{aushik}]
{S.K. S\lowercase{harma}$^\dag$, V\lowercase{irender}$^\S$ \lowercase{and} S.K. K\lowercase{aushik}$^\ddag$ \bigskip \\
	$^{\dag, \ddag}$D\lowercase{epartment} \lowercase{of} M\lowercase{athematics},\\
	K\lowercase{irori} M\lowercase{al} C\lowercase{ollege},\\
	U\lowercase{niversity of} D\lowercase{elhi}, D\lowercase{elhi~110~007}, INDIA.\\
		\emph{$^\dag$}\lowercase{sumitkumarsharma@gmail.com}\\
	\emph{$^\ddag$}\lowercase{shikk2003@yahoo.co.in}\\ 
	\emph{$^\S$}D\lowercase{epartment} \lowercase{of} M\lowercase{athematics},\\
	R\lowercase{amjas}  C\lowercase{ollege},\\
	U\lowercase{niversity of} D\lowercase{elhi}, D\lowercase{elhi~110~007}, INDIA.\\
\emph{$^\ddag$}\lowercase{virender57@yahoo.com}}


\begin{document}
	\maketitle \baselineskip14pt



 \baselineskip12pt
\begin{abstract}
In this article, we introduce and study Riesz bases in a separable quaternionic Hilbert spaces. Some results on Riesz bases in a separable quaternionic Hilbert spaces are proved. It is also proved that a Riesz basis in a separable quaternionic Hilbert space  a frame for the quaternionic Hilbert space. Riesz sequences are defined and equivalence of a Riesz basis and a complete Riesz sequence in a separable quaternionic Hilbert space is proved.
 \end{abstract}\baselineskip14pt


\section{Introduction}
\def\xmn{\cc{u_i}_{i\in \NN}}
\def\hmn{\cc{\nh_{n,i}}_{i=1,2, \cdots, m_n \atop{n\in \NN}}}
\def\h2n{\cc{\nh_{n,i}}_{i=1,2, \cdots, 2n \atop{n\in \NN}}}
\def\tmn{\cc{\mathfrak{T}_{n,i}}_{i=1,2, \cdots, m_n \atop{n\in \NN}}}
\def\Rmn{\cc{\mathfrak{R}_{n,i}: \nh \to \nh_{n,i}}_{i=1,2, \cdots, m_n \atop{n\in \NN}}}
\def\Tmn{\cc{\mathfrak{T}_{n,i}: \nh \to \nh_{n,i}}_{i=1,2, \cdots, m_n \atop{n\in \NN}}}
\def\Umn{\cc{\mathfrak{U}_{n,i}: \nh \to \nh_{n,i}}_{i=1,2, \cdots, m_n \atop{n\in \NN}}}
\def\Tnn{\cc{\mathfrak{T}_{n,i}: \nh \to \nh_{n,i}}_{i=1,2, \cdots, 2n \atop{n\in \NN}}}
\def\ymn{\cc{v_i}_{i\in\NN}}
\def\xxmn{\cc{\la x, x_{n,i}\ra}_{i=1,2, \cdots, m_n \atop{n\in \NN}}}
\def\alphamn{\cc{\alpha_{n,i}}_{i=1,2, \cdots, m_n \atop{n\in \NN}}}
\def\betamn{\cc{\beta_{n,i}}_{i=1,2, \cdots, m_n \atop{n\in \NN}}}
\def\llim{\lim\limits_{n\to \infty}}
\def\S{{S}}
\def\T{\mathfrak{T}_{\cc{m_n}}}
\def\suml{\sum\limits_{i=1}^{\infty}}
\def\mnsum{\sum\limits_{i=1}^{\infty}}
\def\mpsum{\sum\limits_{i=1}^{m_p}}
\def\mqsum{\sum\limits_{i=1}^{m_q}}
\def\elii{\ell^2(\HH)}
\def\sumi{\sum\limits_{n\in \NN}}

Frames for Hilbert spaces, which plays an important role in many applications, were introduced way back in  1952 by Duffin and Schaeffer \cite{DS}
 as a tool to study of non-harmonic Fourier series. Duffin and Schaeffer introduced frames for particular Hilbert spaces of the form $L^2[a,b]$. They defined a frame as

``A sequence $\{x_n\}_{n\in\NN}$ in a Hilbert space $\nh$ is said to be a \emph{frame}
for $\nh$ if there exist constants $A$ and $B$ with $0<A\le B<\infty$
such that
\begin{eqnarray}
A\|x\|^2\le \sum\limits_{n=1}^\infty |\langle
x,x_n\rangle|^2\le B\|x\|^2, \ \ \text{for all} \ x\in \nh."
\end{eqnarray}
Moreover, the positive constants $A$ and $B$, respectively, are called \textit{lower frame bound}
and\textit{upper frame bound}, respectively, for the frame $\cc{x_n}_{n\in\NN}$. Collectively, these are referred as \textit{frame bounds} for the frame $\cc{x_n}_{n\in\NN}$. The inequality
$(1.1)$ is called the \emph{frame inequality} for the frame
$\{x_n\}_{n\in\NN}$. A sequence $\cc{x_n}_{n\in\NN} \subset \nh$ is called a \emph{Bessel sequence} if it satisfies
upper frame inequality in $(1.1)$ i.e. it has upper bound which satisfies the inequality.
A frame $\cc{x_n}_{n\in\NN}$ in $\nh$ is said to be
\begin{itemize}[leftmargin=.5in]
\item \emph{{tight}} if it is possible to choose  $A,\ B$ with $A=B$ satisfying inequality $(1.1)$.
\item \emph{Parseval} if it is possible to choose  $A,\ B$ with $A=B=1$ satisfying inequality $(1.1)$.
\item \emph{exact} if removal of any  $x_n$ leaves the collection $\cc{x_i}_{i\ne n}$ no longer a frame for $\nh$.
\end{itemize}

\medskip

 After more than thirty years, in 1986, Daubechies, Grossmann
and Meyer \cite{DGM}, while studying frames, observed that frames can be used to approximate functions in $L^2({\RR})$. In abstract setting, they observed that a function in $L^2({\RR})$ can be represented as a series in terms of elements of frames, which is similar to bases. Then one can also consider
frames as one of the generalizations of orthonormal bases in Hilbert spaces as  redundant
  frames expansions are more useful and advantageous over basis expansions in a variety of practical
   applications. Now a days, frames are regarded as
an important and integral tool  to study various areas of applications like representation
of signals, characterization of function spaces and other fields
of applications such as signal and image processing \cite{CD}, filter bank theory \cite{BHF},
wireless communications \cite{HP} and sigma-delta quantization \cite{BPY}. For more literature on
frame theory, one may refer to \cite{C1,CH2}.

In recent years, Khokulan, Thirulogasanthar and Srisatkunarajah \cite{KTS} introduced and studied frames for finite dimensional quaternionic Hilbert spaces. Sharma and Virender \cite{SV} study some different types of dual frames of a given frame in a finite dimensional   quaternionic Hilbert space and gave various types of reconstructions with the help of  dual frame. Currently, a lot of work is being carried out in quaternionic Hilbert spaces related to the theory of frames. It is noted that the quaternionic Hilbert spaces are frequrntly used in applied physical sciences especially in pyhsics.
\def\nh{V_R(\HH)}

\subsection{Motivation and Recent Work} 
Frames are used in the study which deals with the  construction of a vector from the elements of the frames. Frames produce the construction of a vector just like the bases produce, but with the relaxation of the uniqueness of the coefficients. Thus frames are used as generalizations of bases. Very recently, Sharma and Goel \cite{SS} introduced and studied frames for separable quaternionic Hilbert spaces  and  Chen, Dang and Qian \cite{CDQ} had studied frames for Hardy
spaces in the contexts of the quaternionic space and the Euclidean space in the Clifford algebra.  Sharma and Virender \cite{SV} study some different types of dual frames of a given frame in a finite dimensional   quaternionic Hilbert space covering the dual part of frames in quaternionic Hilbert spaces. One of the classes of frames is a bit stronger and nicer for the purpose of the computations. Frames of this class are referred to as Riesz frames which are somewhat like Riesz bases. These frames are nicer for computations in applications and are realted closely with the theory of bases. In the present paper, we explore this aspects of Riesz bases in quaternionic Hilbert spaces and try to list the concepts related to Riesz bases in the setup of quaternionic Hilbert spaces. It is arranged in simple and systematic manner which may be referred or may be used whenever there is a requirement.

\subsection {Outline} The paper is organized as follows: In section 2, we give notations and terminology used in the context of quaternionic Hilbert spaces together with some basic theory related to quaternionic Hilbert spaces. In section 2.1, the notations and terminology is given. In section 2.2, some basic definitions and some basic introduction of quaternionic Hilbert spaces are provided. Some fundamental results of functional analysis are listed in the context of quaternionic Hilbert spaces. Section 2.3, frames in quaternionic Hilbert spaces,  comprise the definition of frames and the results related to frames in quaternionic Hilbert spaces. Finally,  Section 3 is devoted to the concept of the Riesz bases in separable quaternionic Hilbert spaces. The basic results related to Riesz bases are listed and proved in this section. This section covers the important aspects of Riesz bases in quaternionic Hilbert spaces. Proper references are provided at the end.

\def\nh{V_R(\HH)}
	
\section{Preliminaries}
\setcounter{equation}{0}
\fontsize{12}{14}

\subsection{Notation and Terminology} Since the quaternions are non-commutative in nature, we have two different types of quaternionic Hilbert spaces, the left  quaternionic Hilbert space and the right quaternionic Hilbert space depending on positions of quaternions. 

Throughout this paper, we will denote $\mathfrak{Q}$ to be a non-commutative field of quaternions,  $I$ be a non empty countable set of indicies, $\H$ be a separable right quaternionic Hilbert space,  by the	term ``right linear operator", we mean a ``right $\HH$-linear operator" and $\mathfrak{B}(\H)$ denotes the set of all bounded (right $\HH$-linear) operators of $\H$:
\begin{eqnarray*}
\mathfrak{B}(\H) := \{T : \H\rightarrow\H: \|T\|<\infty\}.
\end{eqnarray*}

\subsection{Quaternionic Hilbert space}\label{2.2}

The non-commutative field of quaternions $\mathfrak{Q}$ is a four dimensional real algebra with unity. In $\HH$, $0$ denotes the null element and $1$ denotes the identity with respect to multiplication. It also includes three so-called imaginary units, denoted by $i,j,k$. i.e.,
\begin{eqnarray*}
\mathfrak{Q}=\cc{x_0+x_1i +x_2j +x_3k \ :\ x_0,\ x_1,\ x_2,\  x_3\in \RR}
\end{eqnarray*}
where $i^2=j^2=k^2=-1; \ ij=-ji=k; \ jk=-kj=i$ and $ki=-ik=j$. For each quaternion $q=x_0+x_1i +x_2j +x_3k \in \mathfrak{Q}$, the  conjugate of $q$ is
denoted by $\overline{q}$ and is defined as
\begin{eqnarray*}
\overline{q}=x_0-x_1i -x_2j -x_3k \in \mathfrak{Q}.
\end{eqnarray*}
If $q=x_0+x_1i +x_2j +x_3k$ is a quaternion, then $x_0$ is called the real part of $q$ and $x_1i +x_2j +x_3k$ is called the imaginary part  of $q$. The modulus of a quaternion $q=x_0+x_1i +x_2j +x_3k$ is defined as
\begin{eqnarray*}
|q|=(\overline{q}q)^{1/2} = (q\overline{q})^{1/2}= \sqrt{x_0^2 +x_1^2 +x_2^2 +x_3^2 }.
\end{eqnarray*}
For every non-zero quaternion $q=x_0+x_1i +x_2j +x_3k \in \mathfrak{Q}$, there exists a unique inverse $q^{-1}$ in $\mathfrak{Q}$ as
\begin{eqnarray*}
q^{-1}=\dfrac{\overline{q}}{|q|^2 } = \dfrac{x_0-x_1i -x_2j -x_3k }{{x_0^2 +x_1^2 +x_2^2 +x_3^2 }}.
\end{eqnarray*}

\begin{definition}
	A \textit{right quaternionic vector space} $\mathds{V}_R(\HH)$ is a   vector space under right scalar multiplication over the field of quaternionic $\HH$, i.e.,
	\begin{eqnarray}\label{2.1}
	\mathds{V}_R(\HH)\times\HH &\rightarrow& \mathds{V}_R(\HH) \nonumber\\ 
	(u,q)&\rightarrow& uq
	\end{eqnarray}
	and for each $u, v\in\mathds{V}_R(\HH)$ and $p, q\in\HH$, the right scalar multiplication (\ref{2.1}) satisfying the following properties: 
	\begin{eqnarray*}
	&&(u+v)q=uq+vq\\
	&&u(p+q)=up+uq\\
	&&v(pq)=(vp)q.
	\end{eqnarray*}
\end{definition}

\begin{definition}\label{qhs}
	A \textit{right quaternionic pre-Hilbert space} or \textit{right quaternionic inner product space} $\mathds{V}_R(\HH)$ is a right quaternionic vector space together with the binary mapping
	$\langle . | . \rangle : \mathds{V}_R(\HH) \times \mathds{V}_R(\HH) \to \mathfrak{Q}$ (called the \textit{Hermitian quaternionic inner product})
	which satisfies following properties:
	\begin{enumerate}[label=(\alph*)]
		\item $\overline{\langle v_1 | v_2 \rangle} = \langle v_2 | v_1 \rangle$  for all $v_1, v_2 \in \mathds{V}_R(\HH)$.
		\item $\langle v | v \rangle > 0 \ \text{for all} \ \ 0 \ne  v\in \nh.$
		\item $\langle v | v_1 + v_2 \rangle = \langle v | v_1 \rangle + \langle v | v_2 \rangle $  for all $v, v_1, v_2 \in \mathds{V}_R(\HH).$
		\item $\langle v | uq \rangle = \langle v | u \rangle q $  for all $v, u \in \mathds{V}_R(\HH)$ and $ q \in \mathfrak{Q}$.
	\end{enumerate}
\end{definition}

Let $\mathds{V}_R(\HH)$ be right quaternionic inner product space with the Hermitian inner product $\langle .|.\rangle$. ~Define the quaternionic norm $\|.\|:\mathds{V}_R(\HH)\rightarrow\RR^+$ on $\mathds{V}_R(\HH)$ by
\begin{eqnarray}\label{2.2}
\|u\|=\sqrt{\langle u|u\rangle},\ u\in\mathds{V}_R(\HH).
\end{eqnarray}

\begin{definition}
	The right quaternionic pre-Hilbert space  is called a \textit{right quaternionic Hilbert space}, if it is complete with respect to the norm (2.2) and is denoted by $\H$.
\end{definition}

\begin{theorem}[\cite{GMP}) (The Cauchy-Schwarz Inequality] 
	If $\H$ is a right quaternionic Hilbert space then  
	\begin{eqnarray*}
	|\langle u|v\rangle|^2\leq\langle u|u\rangle\langle v|v\rangle, \ \ \text{for all}\ \ u, v \in \H.
	\end{eqnarray*}
	Moreover, a norm as defined in (\ref{2.2}) satisfies the following properties:
	\begin{enumerate}[label=(\alph*)]
		\item $\|uq\|=\|u\||q|$, for all $u\in\H$ and $q\in\HH$.
		\item $\|u+v\|\leq\|u\|+\|v\|$, for all $u, v\in\H$.
		\item $\|u\|=0$ for some $u\in\H$, then $u=0$.
	\end{enumerate}
\end{theorem}

\noindent
For the non-commutative field of quaternions $\HH$,  define the quaternionic Hilbert space $\ell_2(\HH)$ by
\begin{eqnarray*}
\ell_2(\HH) = \bigg\{\{q_i\}_{i\in \NN}\subset \HH : \ \sum_{i\in \NN}|q_i|^2<+\infty \bigg\}
\end{eqnarray*}
under right multiplication by quaternionic scalars together with   the  quaternionic inner product on $\ell_2(\HH)$ defined as
\begin{eqnarray}\label{2.3}
\langle p|q\rangle=\sum_{i\in \NN}\overline{p_i}q_i,\ p=\{p_i\}_{i\in \NN}\ \text{and}\ q=\{q_i\}_{i\in \NN}\in\ell_2(\HH).
\end{eqnarray}
It is easy to observe that $\ell_2(\HH)$ is a right quaternionic Hilbert space with respect to quaternionic inner product (\ref{2.3}).

\begin{definition}[\cite{GMP}] Let $\H$ be a right quaternionic Hilbert Space and  $S$ be a subset of $\H$. Then, define the set:
	\begin{itemize}[leftmargin=.35in]
		\item	$S^{\bot}=\{v\in\H:\langle v|u\rangle=0\ \forall\ u\in S\}.$
		\item $\langle S\rangle$ be the right $\HH$-linear subspace of $\H$ consisting of all finite right $\HH$-linear combinations of elements of $S$.
	\end{itemize}
\end{definition}

\begin{theorem}[\cite{GMP}]\label{2.6t}
	Let $\H$ be a quaternionic Hilbert space and let $N$ be a subset of $\H$
	such that, for $z, z'\in N, \ \langle z|z'\rangle=0$ if $z\neq z'$ and $\langle z|z\rangle=1$. Then the following conditions are equivalent:
	\begin{enumerate}[label=(\alph*)]
		\item For every $u,v\in\H$, the series $\sum_{z\in N}\langle u|z\rangle\langle z|v\rangle$ converges absolutely and
		\begin{eqnarray*}
		\langle u|v\rangle=\sum_{z\in N}\langle u|z\rangle\langle z|v\rangle.
		\end{eqnarray*}
		\item For every $u\in\H$, $\|u\|^2=\sum\limits_{z\in N}|\langle z|u\rangle|^2$.
		\item $N^\bot={0}$.
		\item $\langle N\rangle$ is dense in $H$.
	\end{enumerate}
\end{theorem}

\begin{definition}[\cite{GMP}]
	Every quaternionic Hilbert space $\H$ admits a subset $N$, called \textit{Hilbert basis or orthonormal basis} of $\H$, such that, for $z, z'\in N$, $\langle z|z'\rangle=0$ if $z\neq z'$ and $\langle z|z\rangle=1$ and satisfies all the conditions of Theorem \ref{2.6t}.
\end{definition}
Further, if there are two such sets, then they have the same cardinality. Furthermore, if $N$ is a Hilbert basis of $\H$, then  every $u\in\H$ can be uniquely expressed as
\begin{eqnarray*}
u=\sum_{z\in N}z\langle z|u\rangle,
\end{eqnarray*}
where the series $\sum\limits_{z\in N}z\langle z|u\rangle$ converges absolutely in $\H$.

\begin{definition}[\cite{adler}]
	Let $\H$ be a right quaternionic Hilbert space and $T$ be an operator on $\H$. Then $T$ is said to be
	\begin{itemize}[leftmargin=.25in]
		\item \emph{right $\HH$-linear} if
		$T( v_1\alpha + v_2\beta) = T(v_1)\alpha +  T(v_2)\beta, \ \text{for all} \ v_1, v_2 \in \H \ \text{and} \ \alpha, \beta \in \HH.$
		\item \emph{bounded} if there
		exist $K\ge 0$ such that
		$\pp{T(v)} \le K \pp{v}, \ \text{for all} \ v\in \H.$
	\end{itemize}
\end{definition}

\begin{definition}[\cite{adler}]
	Let $\H$ be a right quaternionic Hilbert space and $T$ be an operator on $\H$. Then the \textit{adjoint operator} $T^*$ of $T$ is defined by
	\begin{eqnarray*}
	\la v|Tu\ra = \la T^*v|u\ra, \ \text{for all} \ u, v \in\H.
	\end{eqnarray*}
	Further, $T$ is said to be \emph{self-adjoint} if $T=T^*$.
\end{definition}

\begin{theorem}[\cite{adler}]
	Let $\H$ be a right quaternionic Hilbert space and $S$ and  $T$ be two bounded right $\HH$-linear operators on $\H$. Then
	\begin{enumerate}[label=(\alph*)]
		\item $T+S$ and $TS\in\mathfrak{B}(\H)$. Moreover:
		\begin{eqnarray*}
		\|T+S\|\leq\|T\|+\|S\|\ \text{and}\ \|TS\|\leq\|T\|\|S\|.
		\end{eqnarray*}
		\item $\la Tv|u\ra = \la v| T^*u \ra$.
		\item $(T+S)^*=T^*+S^*$.
		\item $(TS)^*=S^*T^*$.
		\item $(T^*)^*=T.$
		\item $I^*=I$, where $I$ is the identity operator on $\H$.
		\item If $T$ is an invertible operator then $(T^{-1})^*=(T^*)^{-1}$.
	\end{enumerate}
\end{theorem}


\subsection{Frames in quaternionic Hilbert spaces}

We begin this section with the following definition of frames in a  separable right quaternionic Hilbert space $\H$ as defined in \cite{SS}:
\begin{definition}Let $\H$ be a right quaternionic Hilbert space and $\{u_i\}_{i\in \NN}$ be a sequence in $\H$. Then $\cc{u_i}_{i\in \NN}$ is said to be a \textit{frame} for $V_R(\HH)$, if there exist two finite real quaternions $r_1$ and $r_2$ with $0<r_1\le r_2$  such that
	\begin{eqnarray}\label{3.1}
	r_1\|u\|^2\leq\sum_{i\in \NN}|\langle u_i|u\rangle|^2\leq r_2\|u\|^2, \ \text{for all}\ u\in V_R(\HH).
	\end{eqnarray}
	The finite positive real quaternions $r_1$ and $r_2$, respectively, are called \textit{lower}  and \textit{upper} frame bounds for the frame $\{u_i\}_{i\in \NN}$. The inequality (\ref{3.1}) is called \textit{frame inequality} for the frame $\{u_i\}_{i\in I}$. A sequence $\{u_i\}_{i\in \NN}$ is called a \textit{Bessel sequence} for a right quaternionic Hilbert space $\H$ with bound $r_2$, if $\{u_i\}_{i\in \NN}$ satisfies the right hand side of the inequality (\ref{3.1}).
	A frame $\{u_i\}_{i\in \NN}$ for a right quaternionic Hilbert space $V_R(\HH)$ is said to be
	\begin{itemize}[leftmargin=.25in]
		\item \textit{tight}, if it is possible to choose $r_1$ and $r_2$ satisfying inequality (3.1) with $r_1=r_2$.
		\item \textit{Parseval frame}, if it is tight with $r_1=r_2=1$.
		\item \textit{exact}, if it ceases to be a frame whenever anyone of its element is removed.
	\end{itemize}
\end{definition}

 If $\{u_i\}_{i\in \NN}$ is a Bessel sequence for a right quaternionic Hilbert space $V_R(\HH)$. Then, the \textit{(right) synthesis operator} for $\{u_i\}_{i\in \NN}$ is a right linear operator $T:\ell_2(\HH)\to V_R(\HH)$ defined by
\begin{eqnarray*}T(\{q_i\}_{i\in \NN})=\sum_{i\in \NN}u_iq_i,\ \ \{q_i\}_{i\in \NN}\in\ell_2(\HH).
\end{eqnarray*}
The adjoint operator $T^*$ of right synthesis operator $T$ is called the \textit{(right) analysis operator}. Further,  the  analysis operator $T^*:V_R(\HH)\to \ell_2(\HH)$  is given by
\begin{eqnarray*}
T^*(u)=\{\langle u_i|u\rangle\}_{i\in \NN},\ u\in V_R(\HH).
\end{eqnarray*}
Infact, for $u\in V_R(\HH)$ and $\{q_i\}_{i\in \NN}\in\ell_2(\HH)$, we have
\begin{eqnarray*}
\langle T^*(u)|\{q_i\}_{i\in \NN}\rangle&=&\langle u|T(\{q_i\}_{i\in \NN})\rangle\\
&=&\bigg\langle u\bigg|\sum_{i\in \NN}u_iq_i\bigg\rangle\\
&=&\sum_{i\in\NN}\langle u|u_i\rangle q_i\\
&=&\bigg\langle \{\langle u_i|u\rangle\}_{i\in \NN}, \{q_i\}_{i\in \NN}\bigg\rangle.
\end{eqnarray*}
Thus, we get
\begin{eqnarray*}
T^*(u)=\{\langle u_i|u\rangle\}_{i\in \NN},\ u\in V_R(\HH).
\end{eqnarray*}

Let $\H$ be a right quaternionic Hilbert space and $\{u_i\}_{i\in \NN}$ be a frame for $V_R(\HH)$. Then, the \textit{(right) frame operator} $S:V_R(\HH)\rightarrow V_R(\HH)$ for the frame $\{u_i\}_{i\in \NN}$ is a right linear operator given by
\begin{eqnarray*}
S(u)&=&TT^*(u)\\
&=&\sum_{i\in\NN}u_i\langle u_i|u\rangle,\ u\in \H.
\end{eqnarray*}


\section{Riesz Basis in Quaternionic Hilbert Spaces}

One may recall that, if $N, N' \subset \H$ are Hilbert bases, then there exists a bounded unitary right linear operator $U$ such that $U(N) = N'$. Moreover, if $V$ is bounded unitary right linear operator, then $\cc{V(z)\in \H : z \in N}$ is a Hilbert basis of $H$. We have the following modification of the definition of Riesz basis in a separable right quaternionic Hilbert space

\begin{definition}
	Let $\H$ be a right quaternionic Hilbert space. Then, a \textit{Riesz basis} for $\H$ is a family of the form $U(N)=\cc{U(e_n)}_{n\in\NN}$, where $N=\cc{e_n}_{n\in \NN}$ is an orthonormal basis or Hilbert basis for $\H$ and $U : \H\to \H$ is a bounded right linear bijective operator.
\end{definition}
In the next result, we  prove that a Riesz basis produces another unique Riesz basis which is used in the reconstruction of a vector of the space.

\begin{theorem}\label{Th42}
Let $\H$ be a right quaternionic Hilbert space and $X=\cc{x_n}_{n\in\NN}$ be  a Riesz basis for $\H$. Then $X$ is a Bessel sequence. Further, there exists a unique sequence $Y=\cc{y_n}_{n\in\NN}\subset\H$ such that
	\begin{eqnarray}\label{rz}
	u=\sum_{n\in \NN} x_n \la y_n | u\ra , \ \text{for all } u \in \H.
	\end{eqnarray}
The sequence $Y$ is also a Riesz basis, and the series (\ref{rz}) converges unconditionally for every $u\in \H$.
\end{theorem}

\begin{proof}
As $X$ is a Riesz basis,  $X=\cc{x_n}_{n\in \NN}=\cc{U(e_n)_{n\in \NN}}$, where $\cc{e_n}_{n \in \NN}$ is an orthonormal basis for $\H$ and $U : \H \to \H$ is a bounded right linear bijective operator. So, for every $u\in \H$, we have
\begin{eqnarray*}
U^{-1}(u)= \sum_{n\in \NN}  e_n \left\la e_n | U^{-1}(u) \right\ra
\end{eqnarray*}
Define $Y= \cc{y_n}_{n\in \NN} = \cc{(U^{-1})^*e_n}_{n\in \NN}$. Then, we compute
\begin{eqnarray*}
u &=& UU^{-1} u \\
&=&  \sum_{n\in \NN} U(e_n) \left\la (U^{-1})^*e_n | u \right\ra\\
&=&\sum_{n\in \NN} x_n \left\la y_n| u \right\ra.
\end{eqnarray*}
Since the operator $(U^{-1})^*$ is a bounded right linear bijective operator,  $Y$ is also a Riesz basis. Further, for each $u\in \H$, we have
\begin{eqnarray}
\sum_{n\in \NN} |\la x_n|u \ra|^2  &=& \sum_{n\in \NN} |\la U(e_n)|u \ra|^2 \notag\\
&=& \pp{U^*(u)}^2\notag\\
&\le& \pp{U}^2 \pp{u}^2.
\end{eqnarray}
So, $X$ is a Bessel sequence. Thus, the series (\ref{rz}) converges unconditionally. 

Further,  as $U^{-1}$ is a right linear bounded operator and $\cc{e_n}_{n\in\NN}$ is an orthonormal basis, every  Riesz basis in $\H$ is  also a basis for $\H$. Therefore, if $\cc{y_n}_{n\in \NN}$ and $\cc{z_n}_{n\in \NN}$ are sequences in $\H$ such that for every $u\in \H$
\begin{eqnarray*}
u =\sum_{n\in \NN} x_n\la y_n |u \ra = \sum_{n\in \NN} x_n\la z_n |u \ra,
\end{eqnarray*}
then $\cc{y_n}_{n\in \NN}=\cc{z_n}_{n\in \NN}$. 
\end{proof}

The unique sequence $Y \left(= \cc{y_n}_{n\in \NN} = \cc{(U^{-1})^*e_n}_{n\in \NN}\right) \subset \H$ satisfying (\ref{rz}) is called the \textit{dual Riesz basis} of $X\left(= \cc{x_n}_{n\in\NN} = \cc{U(e_n)}_{n\in \NN}\right)$. Similarly, the dual of $Y$ is
\begin{eqnarray*}
\left\{\left(\left(\left(U^{-1}\right)^*\right)^{-1}\right)^*e_n\right\}_{n\in \NN} &=& \cc{U(e_n)}_{n\in \NN}\\&=& X.
\end{eqnarray*}
Therefore, $X$ and $Y$  are duals of each other. So, we
call them a \textit{pair of dual Riesz bases }in $\H$.

Related to pair of dual Riesz bases, we give the following result.

\begin{theorem}\label{Th43}
	Let  $X=\cc{x_n}_{n\in\NN}$ and $Y=\cc{y_n}_{n\in\NN}$ be  a pair of dual  Riesz basis for a quaternionic Hilbert space $\H$. Then 
	\begin{enumerate}[label=\rm{(\alph*)}]
		\item $X$ and $Y$ forms a  biorthogonal system.
		\item For every $u \in \H$, 
	\begin{eqnarray*}
	u =\sum_{n\in \NN} x_n\la y_n |u \ra = \sum_{n\in \NN} y_n\la x_n |u \ra .
	\end{eqnarray*}
	\end{enumerate}
\end{theorem}

\begin{proof}(a). As $X$ and $Y$ are Riesz bases,   $X=\cc{U(e_n)}_{n\in \NN}$ and $Y= \cc{(U^{-1})^*e_n}_{n\in \NN}$, where  $U : \H \to \H$ is a bounded right linear bijective operator. Then, we have
\begin{eqnarray*}
\la x_i | y_j\ra &=& \la U(e_i) | (U^{-1})^*e_j\ra\\
&=& \la U^{-1}U(e_i) | e_j\ra\\
&=& \left\la e_i |e_j\right\ra = \delta_{ij}, \ \ \text{for each} \ \ i,\ j \in \NN.
\end{eqnarray*}

\noindent
(b). Straight forward.
\end{proof}

In the next result, we show that every Riesz basis in a quaternionic Hilbert space $\H$ is a frame for $\H$.

\begin{theorem}\label{Th42}
	If $\H$ is a right quaternionic Hilbert space and $X=\cc{x_n}_{n\in\NN}$ be  a Riesz basis for $\H$, then $X$ is a frame for $\H$. 
\end{theorem}

\begin{proof}
	In view of (3.6) it is clear that,  $X=\cc{U(e_n)}_{n\in \NN}$  is a Bessel sequence with  bound $\pp{U}^2$, where $U : \H \to \H$ is a bounded right linear bijective operator. Also
	\begin{eqnarray*}
	\pp{u} &=& \pp{(U^*)^{-1} U^* (u)} \\
	&\le& \pp{(U^*)^{-1}} \ \pp{U^*(u)}\\
	&=& \pp{U^{-1}} \ \pp{U^*(u)}.
	\end{eqnarray*}
	Thus, $X$ is the frame with the largest possible lower frame bound  and the smallest
	possible upper frame bound $\dfrac{1}{\pp{U^{-1}}^2}$ and $\pp{U}^2$, respectively. Thus, Riesz bases produce frames for quaternionic Hilbert spaces.
\end{proof}
 Next result is related to regarding Bessel sequences in quaternionic Hilbert spaces.

\begin{theorem}\label{Th35}
Let $\H, \K$ be  right quaternionic Hilbert spaces and let\break$X=\cc{x_n}_{n\in \NN}\subset \H$, $Y=\cc{y_n}_{n\in \NN}\subset \K$. Assume that $Y$ is a Bessel sequence in $\K$ with
bound B,  $X$  is complete in $\H$, and  there exists a positive constant $A$ such that
\begin{eqnarray*}
A \sumi |q_n|^2 \le \left\|{\sumi x_n q_n}\right\|^2
\end{eqnarray*}
for all eventually zero sequences of quaternions $\cc{q_n}_{n\in \NN}$. Then  $U: \text{\rm span } X \to \text{\rm span } Y$ defined by
\begin{align*}
U\left( \sumi x_n q_n \right) = \sumi y_n q_n, \ \ \text{for all eventually zero sequences of quaternions $\cc{q_n}_{n\in \NN}$},
\end{align*}
is a bounded right linear operator with a unique extension to a bounded operator from $\H$ into $\K$, the norm of $U$ as well as its extension is at most $\sqrt{\dfrac{B}{A}}$.
\end{theorem}

\begin{proof}
	For every $x \in \text{span} \ X$, there exist a unique sequence of eventually zero quaternions $\cc{q_n}_{n\in \NN}$ such that 
	$
	x= \sumi x_n q_n.
	$
	Therefore,  $U: \text{span } X \to \text{span } Y$ defined by
	\begin{align*}
	U\left( \sumi x_n q_n \right) = \sumi y_n q_n, \ \ \text{for all eventually zero sequences of quaternions $\cc{q_n}_{n\in \NN},$}
	\end{align*}
	is well-defined and right linear. Further, for an eventually zero quaternions $\cc{q_n}_{n\in \NN}$, we have
	\begin{eqnarray*}
	\left\| U\left( \sumi x_n q_n \right)\right\|^2 &=& \left\| \sumi x_n q_n \right\|^2\\
	&\le& B \sumi |q_n|^2\\
	&\le& \frac{B}{A} \left\| \sumi y_n q_n \right\|^2.
	\end{eqnarray*}
	Thus, $U$ is bounded. Furthermore, as $X$ is   complete, $U$ has an unique	extension to a bounded  operator on $\H$ into $\K$.
\end{proof}

Next, we define a Riesz sequence as follows:

\begin{definition}
	Let $\H$ be a right quaternionic Hilbert space and $X=\cc{x_n}_{n\in\NN}$ be a sequence in $\H$. Then $X$ is said to be a  \textit{Riesz sequence} in $\H$ if there exists positive real quaternions $A$ and $B$ such that
	\begin{eqnarray}\label{e39}
	A  \sumi |q_n|^2 \le \left\| \sumi x_n q_n \right\|^2 \le B  \sumi |q_n|^2 , 
	\end{eqnarray}
	for all eventually zero sequences of quaternions $\cc{q_n}_{n\in \NN}$.
\end{definition}

The next theorem gives an equivalent condition for a Riesz basis in terms of a complete Riesz sequence.

\begin{theorem}\label{Th42}
	Let $\H$ be a right quaternionic Hilbert space. Then  $X=\cc{x_n}_{n\in\NN}$ is a Riesz basis for $\H$ if and only if $X$ is complete Riesz sequence in $\H$.
\end{theorem}

\begin{proof}
	As $X$ is a Riesz basis,   $X=\cc{x_n}_{n\in \NN}=\cc{U(e_n)_{n\in \NN}}$, where $\cc{e_n}_{n \in \NN}$ is an orthonormal  basis for $\H$ and $U : \H \to \H$ is a bounded right linear bijective operator. So, in view of Theorem \ref{Th42}, $X$ is complete in $\H$. Further, for any eventually zero sequences $\cc{q_n}_{n\in \NN} \subset \HH$, we have
	\begin{eqnarray}\label{e37}
	\left\| \sumi x_n q_n \right\|^2 &=& \left\| U\left(\sumi x_n q_n \right)\right\|^2 \notag\\
	&\le&\pp{U}^2\ \left\| \sumi x_n q_n \right\|^2\notag\\
	&=& B  \sumi |q_n|^2.
	\end{eqnarray}
	Further, we compute
	\begin{eqnarray}\label{e38}
	\sumi |q_n|^2 &=& \left\| \sumi e_n q_n \right\|^2 \notag\\
	&=& \left\|U^{-1}U \left(\sumi e_n q_n \right)\right\|^2 \notag\\
	&\le& \pp{U^{-1}}^2 \ \left\| \sumi x_n q_n \right\|^2.
	\end{eqnarray}
	Thus, in view of (\ref{e37}) and (\ref{e38}), we get
	\begin{eqnarray*}
	\dfrac{1}{\pp{U^{-1}}^2} \ \sumi |q_n|^2 \le \ \left\| \sumi x_n q_n \right\|^2 \le \pp{U}^2\ \sumi |q_n|^2.
	\end{eqnarray*}
	\medskip	
	
	Conversely, for each $i\in \NN$, define $U_i : \ell_2(\HH) \to \H$ as
	\begin{eqnarray*}
	U_i(\cc{q_n}_{n\in \NN}) = \sum_{n\le i} x_n q_n
	\end{eqnarray*}
	and $U : \ell_2(\HH) \to \H$ as
	\begin{eqnarray*}
	U(\cc{q_n}_{n\in \NN}) = \sum_{n\in \NN} x_n q_n.
	\end{eqnarray*}
	Then by  (\ref{e39}), each $U_n$ and $U$ are bounded  right linear operators and $\cc{U_n} \to U$. Then by uniform boundedness principle, (\ref{e39}) is true for all $\cc{q_n}_{n\in \NN} \in \ell^2(\HH)$.
	This implies that $X$ is a Bessel sequence with bound $B$. Next, choosing an orthonormal basis $\cc{e_n}_{n\in \NN}$ for $\H$, then by Theorem \ref{Th35}, there exist unique extensions of  the mappings $U(e_n) = x_n, \ n \in \NN$ and $V (x_n) = e_n,  \ n \in \NN$ to    bounded operators on $\H$.  Then, $V U = UV = I$. So $U$ is invertible. Thus, $X$  is a
	Riesz basis for $\H$.
\end{proof}

\begin{corollary}
			Let $\H$ be a right quaternionic Hilbert space and  $X$ be a Riesz sequence in $\H$. Then $X$ is a Riesz basis for $\overline{\text{span}}\ X$.
\end{corollary}

\begin{proof}
	Straight forward.
\end{proof}

\begin{corollary}
	Let $\H$ be a right quaternionic Hilbert space. Then every subfamily of a Riesz sequence is a Riesz sequence.
\end{corollary}

\begin{proof}
	Straight forward.
\end{proof}

\end{document}